\documentclass[11pt, oneside]{amsart}       
\usepackage{geometry}                      
\usepackage[T1]{fontenc}
\usepackage{palatino}

\usepackage{graphicx}                
                            \usepackage{amsthm}    
\newtheorem{theorem}{Theorem}[section]
\newtheorem{definition}{Definition}[section]
\newtheorem{lemma}[theorem]{Lemma}

\newtheorem{corollary}[theorem]{Corollary}
\theoremstyle{remark}
\newtheorem*{remark}{Remark}
\usepackage{tcolorbox}
\newtheorem*{example}{Example}

\title{Derived Geometry and Non-Linear Differential Equations on the Punctured Disc}
\author{EMILE BOUAZIZ}

\begin{document}
\maketitle

\begin{abstract} We study non-linear differential equations on the punctured formal disc by considering the natural derived enhancements of their spaces of solutions. In particular, by appealing to results of the inverse theory in the calculus of variations, we show that a variational formulation of a differential equation is \emph{equivalent} to the residue pairing inducing a (-1)-symplectic form on the derived space of solutions equipped with a certain decoration of its tangent complex. \end{abstract}

\section{introduction}

We work  mainly over the formal punctured disc $\Delta^{*}$. This is defined as the spectrum of $\mathcal{K}:=k((z))$, where $k$ is our base field of characteristic $0$. On $\Delta^{*}$ we will consider differential equations (possibly non-linear, of arbitrary order) $$D(z,y(z),y'(z),y''(z),...)=0,$$ and their spaces of solutions. In fact, we will consider their \emph{derived} spaces of solutions, $\operatorname{dSol}(D)$, which contain some more information. This is a derived ind-scheme cut out by the equations on the Laurent coefficients $y_{j}$ implied by substituting $$y(z)=\sum_{i\in\mathbb{Z}}y_{i}z^{i}$$ into the equation $D=0$. Note that the Laurent series in question is really infinite in both directions as the coefficients are naturally valued in the ring of functions on the ind-scheme of loops into $\mathbb{A}^{1}$. This is responsible for the ind-structure on $\operatorname{dSol}$. We will also consider equations on $\Delta$, the unpunctured disc, with ring of functions $\mathcal{O}:=k[[z]]$. In this case we do not have any ind-structure.

Our main goal is to give a characterisation of \emph{variational} differential equations in terms of the interaction of the residue form on $\mathcal{K}$ and the cotangent complex of the derived space of solutions, $\mathbb{L}(\operatorname{dSol}(D))$. In fact, we will need to decorate $\operatorname{dSol}(D)$ a little, by including as data a lift, along the natural map, of the tangent complex to the loop space of $\mathbb{A}^{1}$. More specifically, we will see that demanding that the residue form induce (in a sense that we will make precise), a \emph{Tate (-1)-symplectic} structure on $\operatorname{dSol}(D)$ is equivalent to $D=0$ having a variational formulation. Of course, one direction is at least morally obvious, given the Euler-Lagrange equations. We remark that there are evident generalisations to collections of differential equations $D_{j}=0$ in functions $y_{i}(z)$, although we will stick to the case of one equation in one dependent variable for ease of exposition.

\section{Basic Properties}
\subsection{Loop Spaces and Tate Derived Ind-Schemes}We will make use of certain quite large algebro-geometrical objects, which unfortunately requires a little technology. The initiated reader should skip this section. We stress that whilst the objects are at first glance somewhat formidable (derived ind-schemes and Tate sheaves on them), the results and computations are simple, and can be understood without a full knowledge of the technology. We will only give a very brief overview of the relevant notions, the reader is referred to \cite{Dr}, \cite{GaiRoz},\cite{He} and \cite{Pr} for much better accounts of the theory. 

We will often say \emph{category} when we often mean $\infty$- such. The category of \emph{derived algebras} is by definition the $\infty$ categorical localisation of commutative differential algebras, $\mathbf{CDGA}_{k}$, at quasi-isomorphisms. We will consider all our derived spaces as living in the category $\mathbf{Pst}_{k}$ of \emph{pre-stacks} over $k$. This is the category of functors from derived algebras to spaces, see \cite{GaiRoz}. Inside $\mathbf{Pst}_{k}$ we have the category of \emph{derived affine schemes}, $\mathbf{dAff}_{k}$, opposite to the category of derived algebras. Further we have the category of \emph{derived schemes}, $\mathbf{dSch}_{k}$ and its category of ind-objects (along ind- systems of closed embeddings) $\mathbf{IndSch}_{k}$, living inside $\mathbf{Pst}_{k}$.
 
For a derived ind-scheme $X$ we have the category of sheaves, $QC(X)$, and its subcategory of perfect objects, $Perf(X)$. Inside $\mathbf{Pro}(QC(X))$ we have the cotangent complex $\mathbb{L}_{X}$, cf. \cite{GaiRoz}. In order to treat the cotangent and tangent complexes on the same footing we will work with the category of \emph{Tate sheaves} on $X$. This category, $\mathbf{Tate}(X)$, is contained in $\mathbf{Pro}(QC(X))$ and contains the images of both $QC(X)$ and $\mathbf{Pro}(Perf(X))$. It is further endowed with a natural duality interchanging these images, cf \cite{He2}. 

\begin{remark} We will only deal with ind- \emph{affine} derived schemes, for which we will have honest presentations as topological $\mathbf{CDGA}_{k}$, where by \emph{topological} we will always mean admitting a neighbourhood basis at $0$ consisting of open ideals. Similarly will treat pro-modules for our algebras as linearly topologized modules and leave implicit that the continuity of the maps constructed ensures a map of pro-systems. \end{remark}

\begin{remark} In the case of $X$ a point, $\mathbf{Tate}(X)$ is the category of \emph{locally linearly compact topological vector spaces}, cf. \cite{Dr}.The prototypial example of such is the $k$-vector space $\mathcal{K}$, endowed with the natural topology on Laurent series.\end{remark}

\begin{definition} We say that an ind-scheme $X$ is \emph{Tate} if the cotangent complex $\mathbb{L}_{X}\in\mathbf{Pro}(QC(X))$ lies in the subcategory $\mathbf{Tate}(X)$ of $X$. \end{definition}

Crucially, if $X$ is a Tate derived ind- scheme then one can speak of \emph{shifted symplectic forms} on $X$. 

The main example of a Tate space for us will be the loop space of the affine line, $\mathcal{L}\mathbb{A}^{1}$. This is represented by the ind- affine scheme $$\operatorname{colim}_{n}\operatorname{spec}k\left[\,y_{i}\,|\,i\geq -n\,\right].$$ The $A$ points of $\mathcal{L}\mathbb{A}^{1}$ are $A((z))$. The cotangent complex, in $\mathbf{Pro}(QC(\mathcal{L}\mathbb{A}^{1}))$, is globally trivial with fibre $\mathcal{K}$, whence is a Tate sheaf. The topological algebra of functions on $\mathcal{L}\mathbb{A}^{1}$ is by definition $$\mathcal{O}(\mathcal{L}\mathbb{A}^{1}):=\lim_{n}k\left[\,y_{i}\,|\,i\geq -n\,\right].$$

 \begin{definition} We define $$\mathcal{A}:=\lim_{n}k\left[\,y_{i}\,|\,i\geq -n\,\right]((z)),$$  noting that we naturally have $y(z),y'(z),...\in\mathcal{A}$. \end{definition}

The reader should bear in mind the following examples, which hopefully help give a feel for Tate symplectic forms on derived ind-schemes. \begin{example} \begin{itemize}\item Consider $\mathbb{A}^{2}$ with the standard symplectic form $dxdy$, then the loop space $\mathcal{L}\mathbb{A}^{2}$ has Tate symplectic form $\sum_{i}dx_{i}dy_{-i}$. \item We could also consider the $-1$-shifted cotangent bundle of $\mathbb{A}^{1}$, with coordinates $x$ and $\xi$ say. Then on the loop space of this we obtain a $-1$-shifted symplectic form $\sum_{i}dx_{i}d\xi_{-i}$.  Note that these are not finite sums of basic two forms, and are topologically convergent as $dx_{i}\rightarrow 0$ as $i\rightarrow -\infty$.\item Let $f$ be a function on the loop space of $\mathbb{A}^{1}$, for example we could take $f$ to be the constant term of $y(z)^{2}$. Then on the derived critical locus of this function we have a Tate $-1$ symplectic form. In fact in this case we obtain the loop space of the derived critical locus of the function $y^{2}$.
\item The loop space of a shifted symplectic derived scheme is Tate shifted symplectic, indeed we can pull-back the two form via the universal map to one on the product of the loop space with $\Delta^{*}$ and then integrate against $\frac{dz}{z}$, in accordance with the well known construction of symplectic forms on mapping spaces in field theory.\end{itemize}\end{example}

\subsection{Solution Spaces}
By a \emph{differential algebra} we will mean a (non-derived) $k$-algebra equipped with a distinguished derivation $\partial$. Maps of differential algebras are defined in the evident way. We will encode differential equations on $\Delta^{*}$ as elements of a certain differential algebra, which we now introduce. \begin{definition} Let $\mathcal{J}$ be the differential algebra defined by $$\mathcal{J}:=\mathcal{K}\left[x_{0},x_{1},x_{2},...\right],$$ equipped with the derivation $\partial_{z}$ which is defined to act on $x_{i}$ by sending it to $x_{i+1}$, and to act on $\mathcal{K}$ as differentiation by $z$.\end{definition}

\begin{remark} \begin{itemize}  \item If $D\in\mathcal{J}$, we write $(D)_{\partial}$ for the differential ideal generated by $D$. \item If $A$ is a $k$-algebra, denote by $\mathcal{J}_{A}$ the differential $A$-algebra $A((z))\left[x_{0},x_{1},x_{2},...\right]$.\item We will consider $A((z))$ as a differential algebra with differential $\partial_{z}$. \item We can speak of differential algebras over a base differential algebra, and maps between these.\end{itemize} \end{remark}

Now there is an obvious way to encode a differential equation $D=0$ as an element of $\mathcal{J}$, indeed we let $x_{i}$ correspond to $y^{(i)}(z)$. We can define then the space of solutions rather cleanly. First we introduce some notation which we will use throughout the sequel. \begin{definition}If $D\in\mathcal{J}$ we will denote by $D_{i}$, the function on $\mathcal{L}\mathbb{A}^{1}$ defined as the $z^{i}$ coefficient of $D(z,y(z),y'(z),y''(z),...)$. We will sometimes suggestively write $D_{i}$ as $$\int D(z,y,y',y'',...)z^{-1-i}dz$$ when we want to stress the analogy with integration. \end{definition}

\begin{remark} Notice that the sum $y(z)$ is not bounded in the Laurent direction, nonetheless the functions $D_{i}$ make sense in the topologial algebra of functions on $\mathcal{L}\mathbb{A}^{1}$. \end{remark}

\begin{definition} The functor, $\operatorname{Sol}(D)$, which sends an algebra $A$ to the set of maps of differential $A((z))$-algebras, $\mathcal{J}_{A}/(D)_{\partial}\rightarrow A((z))$, is referred to as the \emph{space of solutions} of $D$. \end{definition} \begin{remark} Let us note that a $k$-point of $\operatorname{Sol}(D)$, which is by definition a morphism of differential algebras, $\mathcal{J}/(D)_{\partial}\rightarrow \mathcal{K}$, really corresponds to a solution of $D$, in the natural sense. Indeed, compatibility with the differential structures on both sides means that the morphism is determined by the image of $x_{0}$. This image, $\gamma(z)$ say, is now easily seen to satisfy the equation $D$.\end{remark}

\begin{lemma} The functor $\operatorname{Sol}(D)$ is representable by an ind- affine scheme. \end{lemma} \begin{proof} We can easily check that the closed subspace of $\mathcal{L}\mathbb{A}^{1}$ cut out by the equations $D_{i}=0$  represents the desired functor.\end{proof}

As mentioned before, we are interested in the natural derived enhancement of the space $\operatorname{Sol}(D)$, which we will denote $\operatorname{dSol}(D)$. We first note that there is a notion of differential derived algebra, namely a derived algebra equipped with a derivation of cohomologcial degree $0$.   \begin{lemma} The pre-stack $\operatorname{dSol}(D)$, sending a derived algebra $A$ to the space of maps of differential derived algebras over $A((z))$, $\mathcal{J}_{A}/(D)_{\partial}\rightarrow A((z))$, is representable by a derived ind-scheme.  \end{lemma}\begin{proof} We take now the homotopy fibre of the functions $D_{i}$.\end{proof}

\begin{remark} Let us write down explicitly the pro- derived algebra of functions on $\operatorname{dSol}(D)$. For each $n$, we let $\mathcal{O}^{\geq -n}(\operatorname{dSol}(D))$ be the derived algebra freely generated by elements of degree $0$, $y_{i}, i\geq -n$, and elements $\xi_{i},i\geq -n$ of degree $-1$, subject to the relation $\partial(\xi_{j})=D_{j}$. We then take the limit of this family to obtain a topological $\mathbf{CDGA}_{k}$, which models the algebra of functions on $\operatorname{dSol}(D)$. \end{remark}

\begin{remark} We can also consider the sub-space of solutions which extend to the disc $\Delta$. In this case we no longer have ind- objects, and simply have a scheme (resp. derived scheme), denoted $\operatorname{Sol}^{+}(D)$ (resp. $\operatorname{dSol}^{+}(D)$). \end{remark}

\subsection{Tangent Complex and Linearisation} We want to understand the local structure of $\operatorname{dSol}(D)$ near a solution $\gamma(z)\in\mathcal{K}$. To do so let us recall the \emph{linearisation} of $D$ at $\gamma$. This is a linear differential operator on $\mathcal{K}$.\begin{definition} We denote by $\mathcal{L}_{\gamma}(D)$ the linear differential equation on $\Delta^{*}$ defined by $$D(z,\gamma(z)+\epsilon y(z),\gamma'(z)+\epsilon y'(z),...)=0\, \operatorname{mod}\, \epsilon^{2}.$$ We will abusively also write $\mathcal{L}_{\gamma}(D)$ for the corresponding linear differential operator $\mathcal{L}_{\gamma}(D):\mathcal{K}\rightarrow\mathcal{K}$. \end{definition}

We then have the following easy lemma; \begin{lemma} The tangent complex to $\operatorname{dSol}(D)$ at $\gamma$ is equivalent to the length one complex (concentrated in degrees 0 and 1) corresponding to $\mathcal{L}_{\gamma}(D)$, i.e. we have $$\mathbb{T}_{\gamma}(\operatorname{dSol}(D))\cong\left(\mathcal{K}\xrightarrow{\mathcal{L}_{\gamma}(D)}\mathcal{K}\right).$$ A similar result holds for $\operatorname{dSol}^{+}(D)$, with $\mathcal{O}$ in place of $\mathcal{K}$.\end{lemma} \begin{proof}This can be deduced functorially, however we choose to prove it directly by constructing an isomorphism. Recalling the explicit description of functions on $\operatorname{dSol}(D)$, in terms of the variables $y_{i}$ and $\xi_{j}$, the tangent complex is generated by degree $0$ elements $\partial_{y_{i}}$ and degree $1$ elements $\partial_{\xi_{i}}$, with the condition that these are topologically negligible as $i\rightarrow \infty$. We map these to $z^{-i}\in\mathcal{K}$ in their respective cohomological degrees, and it is easy to compute that this intertwines the respective differentials. A similar argument works for $\mathcal{O}$. \end{proof}

In fact we can work globally. Given $D$ we consider now the expression $$\mathcal{L}(D):=D(z,y(z)+\epsilon w(z),y'(z)+\epsilon w'(z),...)=0\, \operatorname{mod}\, \epsilon^{2}.$$ We consider this as a linear differential operator (in the $z$-direction) in the dependent variable $w$, $$\mathcal{L}(D):\mathcal{A}\rightarrow\mathcal{A}.$$  We may also consider this as a length one complex of sheaves on $\mathcal{L}\mathbb{A}^{1}$, noting that the differential operator is linear over $\mathcal{L}\mathbb{A}^{1}$. The generalisation of the above lemma is as follows; \begin{lemma} We have an isomorphism $$\mathbb{T}(\operatorname{dSol}(D))\cong \iota^{*}\left(\mathcal{A}\xrightarrow{\mathcal{L}(D)}\mathcal{A}\right),$$ with $\iota$ the natural map $\operatorname{dSol}(D)\rightarrow\mathcal{L}\mathbb{A}^{1}$. \end{lemma} \begin{proof} The isomorphism is given by the same formula as the lemma above. \end{proof}

\begin{remark} \begin{itemize} \item This allows a computation of the cotangent complex to $\operatorname{dSol}(D)$ as well, indeed it is equivalent to the length one complex (concentrated in degrees -1 and 0 now) corresponding to the formal adjoint to $\mathcal{L}_{\gamma}(D)$, with respect to the residue form. \item Note that this proves that $\operatorname{dSol}(D)$ is Tate. \item We will later be interested in symplectic forms, the above computation makes it clear that it is unnatural to expect one on $\operatorname{dSol}^{+}$, as $\mathcal{O}$ is not self-dual. \end{itemize}\end{remark}

\begin{example}\begin{itemize}\item We work here on the disc $\Delta$ and consider the \emph{Clairaut equation}, which depends on a polynomial $F$. We let $D_{F}$ be the equation $$y=zy'+F(y').$$ We find an $\mathbb{A}^{1}$ of solutions, which is to say a morphism $\mathbb{A}^{1}\rightarrow\operatorname{dSol}^{+}(D_{F})$, mapping $$t\mapsto \gamma_{t}(z):=tz+F(t).$$  We can compute that the linearised equation at $\gamma_{t}$ to be given by $$\mathcal{L}_{\gamma_{t}}(D_{F}): y=(z+F'(t))y',$$ so that the corresponding differential operator on $\mathcal{O}$ gives the tangent complex. It is easy to see then that there is higher tangent cohomology at the solution $\gamma_{t}$ of $D_{F}$, if and only if $t$ is a root of $F'$. We are not aware of an interpetation of these higher tangent classes in terms of the classical geometry of the equation $D_{F}$.
\item Again working over $\Delta$ let us consider $$D: (y')^{2}=4y,$$ with the $\mathbb{A}^{1}$ of solutions $$\gamma_{\lambda}(z):=(z+\lambda)^{2}.$$ The linearised equation at the solution $\gamma_{\lambda}$ is $$(z+\lambda)y'=y$$ and an easy computation again shows that we have isolated points at which there is higher tangent cohomology, this time precisely at $\lambda=0$.
\item We work now on $\Delta^{*}$. Fix elements $a,b\in \mathcal{K}$ and consider the non-linear equation $$D: \,y^{2}=ay''+by',$$ and let $\gamma$ be a solution. We compute the linearised equation as $$\mathcal{L}_{\gamma}(D):\,2\gamma y=ay''+by',$$ and a simple computation tells us that the cotangent complex to this equation is shifted self-dual via the residue form iff $b=a'$, so that the equation is equivalently $$y^{2}=(ay')'.$$ \end{itemize}\end{example}

\subsection{Linear Equations and -1 Symplectic Forms}
Let us now deal with the case of a linear differential equation $D=0$, with the corresponding linear endomorphism of $\mathcal{K}$ denoted $\mathcal{L}_{D}$ to avoid confusion. A consequence of the above argument is that the tangent complex $\mathbb{T}(\operatorname{dSol}(D))$ is free with fibre the complex $\mathcal{L}_{D}:\mathcal{K}\rightarrow\mathcal{K}$. Now let us assume that $D$ is self adjoint with respect to the residue pairing $$(f,g):=\int(fg)dz:=\operatorname{res}_{0}(fgdz),$$ in this case we see that the pairing $\int$ gives an isomorphism $$\mathbb{T}(\operatorname{dSol}(D))\rightarrow\mathbb{L}(\operatorname{dSol}(D))[-1],$$ which in fact we can see comes from a (-1)-symplectic struture $\omega_{\int}$, which is written with respect to the $y$ and $\xi$ variables as $$\omega_{\int}=\sum_{i}dy_{i}d\xi_{-1-i}.$$ We state this now as a lemma, whih we note gives a purely geometric characterisation of self-adjointness - which for example could be generalised immmediately to non-linear equations. \begin{lemma} If $D$ is a linear differential equation on the disc, then it is self-adjoint if and only if the residue form induces a (-1)-symplectic structure on the derived space of solutions $\operatorname{dSol}(D)$. \end{lemma}

 Now, we can in fact say more, indeed we can realise $\operatorname{dSol}(D)$ as a derived critical locus of a function on the loop space $\mathcal{L}\mathbb{A}^{1}$. It can in fact be proven that in the case of a self-adjoint linear differential equation $$D:=\sum_{i}a_{i}(z)y^{(i)}(z)=0,$$ $\operatorname{dSol}(D)$ is the derived critical locus of the function $$\frac{1}{2}\int \sum_{i}a_{i}(z)y(z)y^{(i)}(z)dz$$ on $\mathcal{L}\mathbb{A}^{1}$. It is the goal of the next section to generalise this result to non-linear equations. 

\section{variational calculus}

\subsection{Euler-Lagrange Equations} In this subsection we adapt some standard notions of variational calculus to our purely algebraic setting. 
We begin with an element $D$ of $\mathcal{J}$, and form the function on $\mathcal{L}\mathbb{A}^{1}$, $$\alpha_{D}:=\int D(z,y,y',y'',...)dz.$$ We note that if $D$ is a \emph{total derivative}, which is to say lies in the image of the derivation $\partial_{z}$ of $\mathcal{J}$, then $\alpha_{D}=0$. We expect then that the critical locus of $\alpha_{D}$ should be described by the Euler-Lagrange equations. To formulate these let us recall the variational derivative, acting on $\mathcal{J}$.

\begin{definition} The variational derivative, $\delta:\mathcal{J}\rightarrow\mathcal{J}$ is defined as follows; writing $\partial_{x_{i}}=:\partial_{i}$, we set $$\delta:=\sum_{i}(-1)^{i}\partial_{z}^{i}\partial_{i}.$$ \end{definition}

We have now a purely algebraic version of the Euler-Lagrange equations;

\begin{lemma} (\emph{Euler-Lagrange}.) For an element $D\in\mathcal{J}$, we have an equality of functions on $\mathcal{L}\mathbb{A}^{1}$, $$\frac{\partial}{\partial y_{i}}\int D(z,y,y',y'',...)dz=\int (\delta D)(z,y,y',y'',...)z^{i}dz.$$ \end{lemma} \begin{proof}Let us introduce the expression $w=w(z)=\sum_{i}w_{i}z^{i}$ in indeterminates $w_{j}$, we write $w',w'',...$ in the evident manner. Now a little thought shows that $\frac{\partial\alpha_{D}}{\partial y_{i}}$ is the coefficient of $\epsilon w_{i}$ in the expression $$\int D(z,y+\epsilon w,y'+\epsilon w', y''+\epsilon w'',...)dz.$$ Now we recall the elementary fact that $ab^{(n)}$ is equivalent modulo total derivatives to $(-1)^{n}a^{(n)}b$ and we integrate by parts to remove any derivatives of $w$, obtaining $$\int (\delta D)(z,y,y',y'',...)w(z)dz$$ as the term of order $\epsilon$. Recalling that $w(z)=\sum w_{j}z^{j}$, we see that the coefficient of $w_{j}$ is thus $(\delta D)_{-1-j}$ as required. \end{proof}

\begin{corollary} The residue pairing endows the space $\operatorname{dSol}(\delta D)$ with a Tate (-1)-symplectic form. \end{corollary}\begin{proof} The Euler-Lagrange equations produce an isomorphism $\operatorname{dSol}(\delta D)\cong \operatorname{dcrit}(\alpha_{A})$, where $\operatorname{dcrit}$ denotes the derived critical locus. It is a standard result that a derived critical locus of a function on a smooth space is endowed with a (-1)-symplectic form and this generalises readily to a function on a formally smooth Tate space such as $\mathcal{L}\mathbb{A}^{1}$ . One then checks that the standard symplectic form on $\operatorname{dcrit}(\alpha_{D})$ corresponds to the form $\omega_{\int}:=\sum dy_{i}d\xi_{-1-i}$ on $\operatorname{dSol}(\delta D)$. \end{proof}

The main result of this note is a sort of converse to the above. First a couple of definitions; \begin{definition} We call a \emph{framing} of a derived ind-scheme equipped with a map to $\mathcal{L}\mathbb{A}^{1}$, a lift of the tangent complex to $\mathcal{L}\mathbb{A}^{1}$ and we consider $\operatorname{dSol}(D)$ as framed by the length one complex of sheaves $\mathcal{A}\xrightarrow{\mathcal{L}(D)}\mathcal{A}$.\end{definition} \begin{definition} We will say that the residue form induces a (-1)-symplectic form on $\operatorname{dSol}(D)$ if $\operatorname{dSol}(D)$ admits a (-1)-symplectic form which lifts to the framing, on which it acts as the residue form.\end{definition}  \begin{theorem} Let $D=0$ be a differential equation on $\Delta^{*}$ such that the residue pairing endows $\operatorname{dSol}(D)$ with a Tate (-1)-symplectic form. Then $D=0$ admits a variational formulation, namely there is an $A\in\mathcal{J}$ with $\delta A=D$. \end{theorem}

\begin{proof} With respect to our model with variables $y_{i}$ and $\xi_{j}$, we see that we are assuming $\sum dy_{i}d\xi_{-1-i}$ is a Tate symplectic form. This will hold precisely if it is closed for the internal differential on $\mathcal{O}(\operatorname{dSol}(D))$, with respect to our usual model. Such is true precisely if we have the following conditions, which we think of as a sort of \emph{integrability}, we must have that for all $i,j$, $$\frac{\partial D_{-1-i}}{\partial y_{j}}=\frac{\partial D_{-1-j}}{\partial y_{i}}.$$ We note now that this easily implies that there is some $A\in \mathcal{O}(\mathcal{L}\mathbb{A}^{1})$ such that $\partial_{y_{i}}A=D_{i}$, simply as de Rham cohomology of $\mathcal{L}\mathbb{A}^{1}$ vanishes. This is not good enough however, as it does not produce a variational formulation.

To produce such, we must appeal to the inverse theory in the calculus of variations. In particular it is known (cf. \cite{Kru}) that we must check that the following \emph{Helmholtz integrability conditions} hold;  for all $l\geq 1$ we have the equality which we denote $H_{l}(D)$; $$(1+(-1)^{l+1})\partial_{l}D=\sum_{k>l}(-1)^{k}\binom{k}{l}\partial_{z}^{k-l}\partial_{k}D.$$

Now we rewrite the equations $$\frac{\partial D_{-1-i}}{\partial y_{j}}=\frac{\partial D_{-1-j}}{\partial y_{i}},$$ using the Euler-Lagrange equations, as $$\int (z^{j}\delta(z^{i}D)-z^{i}\delta(z^{j}D))=0,$$ for all $i,j$.

 For brevity assume that $D$ is of order $2$. We compute that the above integrand is given by $$(j-i)z^{i+j-1}\partial_{y'}D+2(i-j)z^{i+j-1}(\partial_{y''}D)'+(i(i-1)-j(j-1))z^{i+j-2}\partial_{y''}D,$$ we now integrate the middle term by parts so that we obtain a common factor of $z^{i+j-1}$. We deduce that for all $i,j$ we have $$\int z^{i+j-1}(\partial_{y'}D-(\partial_{y''}D)')dz=0,$$ whence we see that the Helmholtz condition, $\partial_{y'}D=(\partial_{y''}D)'$ holds. 

 In general we argue as follows; we first integrate by parts so that we have a common factor of $z^{i+j-1}$. This gives the first Helmholtz equation $H_{1}$, which we note implies that $\partial_{y'}D$ is a total derivative in $z$, this allows us then substitute for $\partial_{y'}D$ and then integrate by parts until we obtain a common factor of $z^{i+j-2}$, from which we obtain the second Helmholtz condition, and so on. Note that there is a subtlety, namely that we can only perform the above integration by parts for generic values of $i,j$. Indeed, we cannot integrate $z^{-1}$. Nonetheless it is easy to see that generic vanishing of $D_{n}$, which is to say vanishing for all but finitely many $n$, implies vanishing of $D$, assuming $D$ is non-constant in the dependent variable $y$.

Now according to a standard result in the inverse theory of calculus of variations, we can construct a Lagrangian, cf the results of chapter 4 of \cite{Kru}. In fact we can construct one explicitly according to the recipe of Vainberg-Tonti (again cf. chapter 4 of \cite{Kru}). Such is given by $$\mathcal{L}:=y\int_{0}^{1}D(z,ty,ty',ty'',...)dt,$$ where we interpret the definite integration as a linear form in the evident manner.

\end{proof}

 \begin{remark}\begin{itemize} There are a couple of notable aspects of this result. \item It is crucial that we work on a punctured disc $\Delta^{*}$. The corresponding theorem is not true on $\Delta$. \item It is not the case that $\operatorname{dSol}(D)$ admitting a (-1)-symplectic form implies that $D$ admits a variational formulation. Indeed there are linear differential operators $D$ whose adjoint operator $D^{*}$ is simply $-D$, for example $D=ay'+\frac{a'}{2}y$. In this case $\operatorname{dSol}$ is certainly (-1)-symplectic although $D$ does not admit a variational formulation.  \end{itemize}\end{remark}

\end{document}